\documentclass[a4paper,11pt]{article}
\usepackage[utf8]{inputenc}
\usepackage{amsmath}
\usepackage{latexsym}
\usepackage{amssymb,amsfonts,amsmath,mathrsfs}
\usepackage{verbatim}
\usepackage{amsthm}
\usepackage{amsfonts}
\usepackage{amssymb}
\usepackage{amstext}
\usepackage{dsfont}
\usepackage{graphicx}
\usepackage{color}
\usepackage[colorlinks]{hyperref}
\usepackage{epigraph}
\usepackage[left=3cm,top=2cm,bottom=3.5cm,right=2.5cm,nohead]{geometry}
\usepackage{tikz}
\newtheorem{theorem}{Theorem}[section]
\newtheorem{lemma}[theorem]{Lemma}
\newtheorem{proposition}[theorem]{Proposition}

\newtheorem{definition}[theorem]{Definition}

\date{}
\title{Rigidity of the $\operatorname{Sine}_{\beta}$ process}
\author{Reda Chhaibi, Joseph Najnudel}

\begin{document}
\maketitle

\begin{abstract}
We show that the $\operatorname{Sine}_{\beta}$ point process, defined as the scaling limit of the Circular Beta Ensemble when the dimension goes to infinity, and generalizing the determinantal sine-kernel process, is rigid in the sense of Ghosh and Peres: the number of points in a given bounded Borel set $B$ is almost surely equal to a measurable function of the position of the points outside $B$.
\end{abstract}

\section{Introduction}
If $E$ is a complete separable metric space, we can define a point process on $E$ as a random purely atomic Radon measure $X$ on $E$, which can be viewed as a random locally finite collection of points in $M$, with possible repetitions, the order of the points being irrelevant. Classical results on the theory of point processes are for example given in the book by Daley and Vere-Jones \cite{DVJ03}. If $X$ is a point process on $E$, and if $B$ is a Borel set of $E$, we can consider the random variable $X(B)$, corresponding to the number of points lying in $B$. Of course, this random variable does not tell about  the exact position of the $X(B)$ points of $X$ inside $B$. This information is given by the $\sigma$-algebra $\Sigma_B$, generated by all the variables $X(A)$ for Borel sets $A$ included in $B$. 

It has been observed that some point processes satisfy the unusual property that the number of points in any Borel set is uniquely determined by the points in the complement of this set. For example, deterministic point processes, or periodic  point processes for $E = \mathbb{R}$, obviously satisfy this property, whereas Poisson point processes with non-zero intensity do not. 
The following definition has been introduced by Ghosh \cite{G15} (see also Ghosh and Peres \cite{GP17}): 
\begin{definition} \label{rigidity}
A point process $X$ on a complete separable metric space $E$ is {\it rigid} if and only if for all bounded Borel subsets $B$ of $E$, the number of points $X(B)$ in $B$ is measurable with respect to the $\sigma$-algebra $\Sigma_{E \backslash B}$.
\end{definition} 
This notion has been previously studied, with a different name, by Holroyd and  Soo \cite{HS13}. 
In \cite{G15}, Ghosh shows that the {\it determinantal sine-kernel process} is rigid: it is the first non-trivial example of such process. The determinantal sine-kernel is a real-valued point process ($E = \mathbb{R}$), defined as follows: it has no multiple points and its $m$-point correlation function  at $x_1, \dots, x_m$  is equal to $\operatorname{det} ((\mathbb{S}(x_j,x_k))_{1 \leq j, k \leq m})$, where 
the {\it sine-kernel} $\mathbb{S}$ is defined by 
$$\mathbb{S} (x,y) = \frac{ \sin(\pi(x-y))}{\pi(x-y)}.$$
As proven by Dyson (see \cite{D621}, \cite{D622}, \cite{D623}), the determinantal sine-kernel process can be obtained as the limit, when $n$ goes to infinity, of the point process of the eigenangles of a Haar-distributed unitary matrix (the Circular Unitary Ensemble) multiplied by $n/2 \pi$. The sine-kernel process is also a scaling limit for many other matrix ensembles, including the Gaussian Unitary Ensemble, for which the matrix is Hermitian, invariant by unitary conjugation, and has complex Gaussian entries. It is also conjectured that the sine-kernel process is related to the distribution of the zeros of the Riemann zeta function (see Montgomery \cite{M73}, Rudnick and Sarnak \cite{RS94}). 

Many determinantal processes have been studied in the literature, in particular in relation with random matrix theory: a general presentation of these processes is for example given by Soshnikov in \cite{S00}. Besides the sine-kernel process,    other determinantal processes have been proven to be rigid: it is the case for the determinantal processes with Airy kernel and with Bessel kernel (see Bufetov \cite{B16}), and for determinantal processes associated to de Branges spaces of holomorphic functions (see Bufetov and Shirai \cite{BS17}). Some two-dimensional point processes are also proven to be rigid, including the infinite Ginibre ensemble and the set of zeros of some Gaussian analytic functions (see Ghosh and Peres \cite{GP17}). Rigidity of other two-dimensional determinantal processes has been studied in papers by Bufetov and Qiu (see \cite{BQ15}, \cite{BQ17} and \cite{BQ172}). 

The probability density of the eigenvalue distribution of the Circular Unitary Ensemble is proportional to a constant times the product of the squares of the  mutual distances between the points. If we replace the squares by $\beta$ powers ($\beta > 0$), we get the {\it Circular Beta Ensemble}, which is also the spectrum of random matrix ensembles constructed by Killip and Nenciu \cite{KN04}. Using these random matrices,  Killip and Stoiciu \cite{KS09} prove that the arguments of the Circular Beta Ensemble, multiplied by the dimension and divided by $2 \pi$, tend to a limiting point process, called $\operatorname{Sine}_{\beta}$ process, which corresponds to the determinantal sine-kernel process for $\beta = 2$.

In \cite{VV09}, V\'alko and Vir\'ag show that the  $\operatorname{Sine}_{\beta}$ process is also the scaling  limit of the Gaussian Beta Ensemble, generalizing the Gaussian Unitary Ensemble, and corresponding to a model of random tridiagonal matrices introduced  by Trotter \cite{T84}, and by Dumitriu and Edelman \cite{DE02}. A description of the $\operatorname{Sine}_{\beta}$ in terms of the hyperbolic Brownian motion is given in \cite{VV09}. In \cite{VV17}, V\'alko and Vir\'ag construct a Hermitian operator whose spectrum forms a $\operatorname{Sine}_{\beta}$ process. 

Since the determinantal sine-kernel process is rigid, it is natural to expect that it is also the case for the $\operatorname{Sine}_{\beta}$ process. In this paper, we give a proof of this result, which gives, to our knowledge, the first example of a rigid process for which no explicit formula is known for the correlation functions. 
\begin{theorem}
The $\operatorname{Sine}_{\beta}$ point  process, defined as the weak limit, when $n$ goes to infinity, of the point process of the eigenangles of the $n$-dimensional Circular Beta Ensemble, multiplied by $n/2 \pi$, is rigid in the sense of Definition \ref{rigidity}. 
\end{theorem} 
The global strategy of the proof is similar to what has been done for the rigid process considered before: we find functions $f$ which are arbitrarily close to $1$ in an  interval $I$, and whose corresponding linear statistics have arbitrarily small variance when they are applied to the   $\operatorname{Sine}_{\beta}$ point  process $X$.
This gives an approximation of the number of points in $I$ in terms of the sum of $f$ at points of $X$ outside $I$. 

The variance of the linear statistics are estimated by using bounds on the variance of $\operatorname{Tr}(M^k)$, where $M$ is a matrix corresponding to the Circular Beta Ensembles and $k$ is not too large, and by letting the dimension of the $C \beta E$ tending to infinity. 

The details of the proof are given in the next section. We use the following notation: $A \ll_{\beta} B$ means that there exists $C(\beta) > 0$ depending only on $\beta$, such that $|A| \leq C(\beta) B$. 

\subsection*{Acknowledgements}
The authors are grateful to Sasha Bufetov and Yanqi Qiu for fruitful discussions as well as introducing them to the concept of rigidity of point processes.

\section{Proof of the main theorem} 
In all this section $\beta > 0$ is a fixed parameter. 
The main estimate on the  $C\beta E$ which is used in our proof comes from a paper by Jiang and Matsumoto 
\cite{JM15}. Corollary 2, (a) of this paper (applied to the partition with unique element $k$) gives the following: 
\begin{lemma}
Let $M$ be a unitary matrix whose spectrum corresponds to the $C \beta E$ of dimension $n$. Then, for $0 < k \leq n/2$, 
$$\mathbb{E}[ |\operatorname{Tr}(M^{k})|^2] \ll_{\beta} k.$$
\end{lemma} 
From this lemma, we can prove the following: 
\begin{proposition}
Let $X_n$ be the set of points following the $C \beta E$ of dimension $n$ and let $f$ be a smooth function from the unit circle $\mathbb{U}$ to $\mathbb{C}$, whose  Fourier transform  vanishes at zero and outside the interval $[-n/2,n/2]$. Then 
$$\mathbb{E} \left[ \left| \sum_{z \in X_n} f(z) \right|^2 \right] \ll_{\beta} \sum_{k \in \mathbb{Z}} |k| \hat{f}(k),$$
uniformly on $n$. Here, the Fourier transform $\hat{f}$ is defined by the expansion: 
$$f(z) = \sum_{k \in \mathbb{Z}} \hat{f}(k) z^k.$$
\end{proposition}
\begin{proof}
We have 
$$\mathbb{E} \left[ \left| \sum_{z \in X_n} f(z) \right|^2 \right]
= \mathbb{E} \left[ \sum_{z, z' \in X_n} \sum_{k \in \mathbb{Z}}  \hat{f}(k) z^k
 \sum_{k' \in \mathbb{Z}}  \overline{\hat{f}(k') (z')^{k'} } \right].$$
 Everything is integrable since $\hat{f}$ has finite support, so we can apply Fubini's theorem:
 $$ \mathbb{E} \left[ \left| \sum_{z \in X_n} f(z) \right|^2 \right]
 = \sum_{k, k' \in \mathbb{Z}}\hat{f}(k) \overline{\hat{f}(k')} 
 \mathbb{E} \left[ \sum_{z, z' \in X_n} z^k \overline{(z')^{k'}}\right].$$
 If $k \neq k'$, the last expectation is multiplied by $u^{k-k'}$ if we multiply all the points of $X_n$ by $u \in \mathbb{U}$. On the other hand, it should be conserved since the law of $X_n$ is rotationally invariant. Hence, the expectation is zero: 
 $$\mathbb{E} \left[ \left| \sum_{z \in X_n} f(z) \right|^2 \right]
  = \sum_{k \in \mathbb{Z}}|\hat{f}(k)|^2  \mathbb{E} \left[ \sum_{z, z' \in X_n} z^k \overline{(z')^{k}}\right]$$
  Now, 
  $$ \sum_{z, z' \in X_n} z^k \overline{(z')^{k}} 
  = \left| \sum_{z \in X_n} z^k \right|^2 = |\operatorname{Tr}(M^{k})|^2$$
  where $M$ is a matrix whose eigenvalues form the set $X_n$. 
  We conclude by using the previous lemma. 
\end{proof}

In order to take a limit when $n$ goes to infinity, it is useful to translate the result above in terms of the 
 renormalized arguments of the points of $X_n$. 
For a function $f$ from $\mathbb{R}$ to $\mathbb{C}$ in the Schwartz space, we introduce its Fourier transform: 
$$\hat{f} (\lambda) = \int_{-\infty}^{\infty} f(t) e^{-2i \pi \lambda t} dt.$$
which  should not be confused with the Fourier transform of a function from $\mathbb{U}$ to $\mathbb{C}$, even if the two notions of Fourier transform are denoted in the same way  in this paper. 
\begin{proposition}
Let $f$ be a function from $\mathbb{R}$ to $\mathbb{C}$ in the Schwartz space, such that 
its Fourier transform vanishes outside the interval $[-1/2,1/2]$. 
Let $X_n$ be the set of points following the $C \beta E$ of dimension $n$. 
For $n \geq 1$, let $E_n$ be the set of all possible  determinations of the arguments of the points of $X_n$, multiplied by $n/2 \pi$ (in particular $E_n$ is a $n$-periodic set of points).  Then, 
$$\mathbb{E} \left[ \left|\sum_{x \in E_n} f(x) - \int_{-\infty}^{\infty} f(t) dt \right|^2 \right] 
\ll_{\beta} \int_{\mathbb{R}} |[x]_n| |\hat{f} ([x]_n)|^2 dx,$$
where $[x]_n$ is $1/n$ times the integer part of $nx$ (the quantity just above is then a Riemann sum).
\end{proposition}

\begin{proof}
We have 
$$\sum_{x \in E_n} f(x) - \int_{-\infty}^{\infty} f(t) dt
= \sum_{z \in X_n} g(z),$$
where $X$ follows the $C \beta E$ of dimension $n$ and 
$$g(e^{2 i \pi \theta}) = \sum_{m \in \mathbb{Z}} f( n (\theta + m)) - \frac{1}{n} \int_{-\infty}^{\infty} f(t) dt.$$
Since $f$ is in the Schwartz space, $g$ is smooth by dominated convergence. Moreover, 
since
$$g(z) = \sum_{k \in \mathbb{Z}} \hat{g}(k) z^k,$$
we have 
\begin{align*}
\hat{g}(k) & = \int_0^1 g(e^{2 i \pi \theta}) e^{-2 i \pi k \theta} d \theta
\\  & =  \int_0^1  d \theta \, e^{-2 i \pi k \theta} \sum_{m \in \mathbb{Z}} f( n (\theta + m))
- \int_0^1 d \theta \, e^{-2 i \pi k \theta} \frac{1}{n} \int_{-\infty}^{\infty} f(t) dt.
\\ & = \int_{-\infty}^{\infty} e^{-2 i \pi k \theta} f( n \theta) d \theta
- \frac{\mathds{1}_{k =0} }{n} \int_{-\infty}^{\infty} f(t) dt
 \\ & = \frac{1}{n} (\hat{f}(k/n) - \mathds{1}_{k =0} \hat{f}(0)) 
  = \frac{\mathds{1}_{k \neq 0}}{n} \hat{f}(k/n).
\end{align*} 
In particular, the Fourier transform of $g$ is equal to zero at zero and outside $[-n/2,n/2]$, and then we can apply the previous proposition. 
We get 
\begin{align*}
\mathbb{E} \left[ \left|\sum_{x \in E_n} f(x) - \int_{\infty}^{\infty} f(t) dt \right|^2 \right] 
& \ll_{\beta} \sum_{k \in \mathbb{Z}} |k| |\hat{g}(k)|^2
 = \sum_{k \in \mathbb{Z}} |k| |\hat{f}(k/n)/n|^2
\\ & = \frac{1}{n} \sum_{k \in \mathbb{Z}} |k/n||\hat{f}(k/n)|^2,
\end{align*}
which gives the desired Riemann sum. 
\end{proof}
Passing to the limit when $n$ goes to infinity, we deduce a bound on the variance of the linear statistics of the $\operatorname{Sine_{\beta}}$ process in terms of the $H^{1/2}$ norm of the test function $f$, as soon as the Fourier transform of $f$ is supported in $[-1/2,1/2]$.  
\begin{proposition}
Let $f$ be a function from $\mathbb{R}$ to $\mathbb{R}$ in the Schwartz space, whose Fourier transform is supported in 
$[-1/2,1/2]$. 
Let $E$ be the set of points of a  $\operatorname{Sine}_{\beta}$ process. We have 
$$\mathbb{E} \left[\left| \sum_{x \in E} f(x) - \int_{-\infty}^{\infty} f(t) dt \right|^2 \right] 
\ll_{\beta} \int_{\mathbb{R}} |x| |\hat{f} (x)|^2 dx.$$
\end{proposition}
\begin{proof}
Let $f = f_1 + f_2$, where $f_1$ is smooth with compact support. 
By the convergence in law of the renormalized $C \beta E$ towards the $\operatorname{Sine}_{\beta}$, we have the convergence in law: 
$$\sum_{x \in E_n} f_1(x) 
\underset{n \rightarrow \infty}{\longrightarrow} \sum_{x \in E} f_1(x),$$
i.e. for $\lambda \in \mathbb{R}$, 
$$\mathbb{E} \left[ \exp \left( i \lambda \sum_{x \in E_n} f_1(x) \right) \right]
\underset{n \rightarrow \infty}{\longrightarrow}
\mathbb{E} \left[ \exp \left( i \lambda \sum_{x \in E} f_1(x) \right) \right].$$
Using the Lipschitz property of the complex exponential and the triangle inequality, we deduce 
\begin{align*}
& \underset{n \rightarrow \infty}{\lim\sup} 
\left|\mathbb{E} \left[ \exp \left( i \lambda \sum_{x \in E_n} f(x) \right) \right]
- 
\mathbb{E} \left[ \exp \left( i \lambda \sum_{x \in E} f(x) \right) \right] \right|
\\ & \leq |\lambda| \, \underset{n \rightarrow \infty}{\lim\sup} \,
\mathbb{E} \left[  \sum_{x \in E_n} |f_2(x) | + \sum_{x \in E} |f_2(x) | \right]
\\ & = 2 |\lambda| \, \int_{-\infty}^{\infty} |f_2(t)| dt,
\end{align*}
the last equality coming from the fact that the one-point correlation function of the point processes $E_n$ and $E$ is equal to one. 
Since the integral of $|f_2|$ can be arbitrarily small (take for $f_1$ a smooth truncation of $f$, equal to $f$ in a sufficiently large bounded interval), the upper limit is zero and we have the convergence in law:
$$\sum_{x \in E_n} f(x) 
\underset{n \rightarrow \infty}{\longrightarrow} \sum_{x \in E} f(x),$$
which imples the convergence in law
$$\left|\sum_{x \in E_n} f(x) - \int_{-\infty}^{\infty} f(t) dt \right|^2
\underset{n \rightarrow \infty}{\longrightarrow} \left|\sum_{x \in E} f(x) - \int_{-\infty}^{\infty} f(t) dt \right|^2.$$
If $A > 0$, we deduce 
$$\mathbb{E} \left[ \left| \sum_{x \in E_n} f(x) - \int_{-\infty}^{\infty} f(t) dt \right|^2 \wedge A \right]
\underset{n \rightarrow \infty}{\longrightarrow} \mathbb{E} 
\left[ \left| \sum_{x \in E} f(x) - \int_{-\infty}^{\infty} f(t) dt \right|^2 \wedge A \right],$$ 
$$ \underset{n \rightarrow \infty}{\lim\inf} \, \mathbb{E} \left[ \left| \sum_{x \in E_n} f(x) - \int_{-\infty}^{\infty} f(t) dt \right|^2 \right]
 \geq \mathbb{E} 
\left[ \left| \sum_{x \in E} f(x) - \int_{-\infty}^{\infty} f(t) dt \right|^2 \wedge A \right],$$
and by letting $A \rightarrow \infty$, 
$$\mathbb{E} \left[ \left| \sum_{x \in E} f(x) - \int_{-\infty}^{\infty} f(t) dt \right|^2  \right]
\leq \underset{n \rightarrow \infty}{\lim\inf} \, \mathbb{E} \left[ \left| \sum_{x \in E_n} f(x) - \int_{-\infty}^{\infty} f(t) dt \right|^2 \right].$$
By the previous proposition, it is enough to show that 
$$\int_{\mathbb{R}} |[x]_n| |\hat{f} ([x]_n)|^2 dx \underset{n \rightarrow \infty}{\longrightarrow}
 \int_{\mathbb{R}} |x| |\hat{f} (x)|^2 dx.$$
 Since 
 $\hat{f}$ is smooth ($f$ is in the Schwartz space) with support included in $[-1/2,1/2]$, this last convergence is 
 just the convergence of the Riemann sums of an integral. 
\end{proof}
The next proposition shows that we can apply the result just above to functions which enjoy suitable properties 
for the proof of the rigidity of the point process $E$: 
\begin{proposition}
Let $R > 0$. There exists $f$ in the Schwartz space, with Fourier transform supported on $[-1/2,1/2]$, such that  
$$||f||_{H^{1/2}} := \left(\int_{\mathbb{R}} |x| |\hat{f} (x)|^2 dx \right)^{1/2}$$
and 
$$\underset{t \in [-R,R]}{\sup} |f(t) -1|$$
are arbitrarily small. 
\end{proposition}
\begin{proof}
Let $f$ be the inverse Fourier transform of a nonnegative smooth function supported on $[-1/2,1/2]$, normalized in such a way that $f(0) = 1$. Then, $f$ is in the Schwartz space, with Fourier transform supported on $[-1/2,1/2]$. For $L > 1$, we define 
$$f_L(t) :=\frac{1}{2}( f(t) + f(t/L)).$$
This function is still in the  Schwartz space, with Fourier transform supported on $[-1/2,1/2]$, and $f_L(0) = 1$. 
We have 
$$4 ||f_L||_{H^{1/2}}^2 =  \int_{\mathbb{R}}|x| |\hat{f} (x)|^2 dx
+ \int_{\mathbb{R}} |x| |L \hat{f} (Lx)|^2 dx + 2 \int_{\mathbb{R}} |x| \Re \left(  \hat{f} (x)
\overline{L  \hat{f} (Lx) } \right) dx.$$
The first integral is $|f||_{H^{1/2}}^2$, the second has the same value by a change of variable. 
The last integral is supported in $[-1/2,1/2]$, the integrand being uniformly bounded
by the supremum of $|\hat{f}|$ times the supremum of $|y| |\hat{f}(y)|$ (take $y = Lx$). For fixed $x$, it is equal to 
zero when $L$ is large enough depending on $x$. By dominated convergence,  
$$\int_{\mathbb{R}} |x| \Re \left(  \hat{f} (x)
\overline{L  \hat{f} (Lx) } \right) dx
\underset{L \rightarrow \infty}{\longrightarrow} 0.$$
We deduce that for $L$ large enough (depending on $f$), 
$$||f_L||_{H^{1/2}}^2 \leq 0.51 ||f||_{H^{1/2}}^2.$$
Iterating this process gives an arbitrarily  small value of  $||f||_{H^{1/2}}^2$. 
Since $f$ is smooth and equal to $1$ at zero, we can do a final replacement of $f$ by $f_L$ for $L$ large (the support of the Fourier transform is still included in $[-1/2,1/2]$ and the $H^{1/2}$ norm is not changed) in order to get 
a uniformly small value of  $|f - 1|$ on $[-R,R]$. 
\end{proof}
We have now all the ingredients needed to finish the proof of the main theorem. 

{\it Proof of the main theorem:} Let $E$ be the set of points in a $\operatorname{Sine}_{\beta}$ process. Let $B$ be a  bounded Borel set of $\mathbb{R}$, included in an interval $[-R,R]$, $R > 0$. 
Let $(f_p)_{p \geq 1}$ be a sequence of  functions in the Schwartz space,  with Fourier transform supported on $[-1/2,1/2]$,  such that 
$$||f_p||_{H^{1/2}} \leq 2^{-p}, \underset{t \in [-R,R]}{\sup} |f_p(t) -1| \leq 2^{-p}.$$
If $\mathcal{N}_C$ denotes the number of points of  $E$ in the subset $C$ of $\mathbb{R}$, we  have 
$$ \mathcal{N}_{[-R,R]}=   \left(\sum_{x \in E} f_p(x) - \int_{-\infty}^{\infty} f_p(t) dt \right)
- \sum_{x \in E, |x| \leq   R} [f_p(x) - 1] + 
\int_{-\infty}^{\infty} f_p(t) dt - \sum_{x \in E, |x|>  R} f_p(x)
.$$
The $L^2$ norm of the term into parentheses is dominated by $||f_p||_{H^{1/2}} \leq 2^{-p}$.
The $L^1$ norm of the sum involving $f_p- 1$ is (because of the one-point correlation function),
at most $2R \cdot 2^{-p}$. 
 Hence, these two terms almost surely tend to zero when $p$ goes to infinity. We deduce that 
almost surely, 
$$\mathcal{N}_{[-R,R]}  = \underset{p \rightarrow \infty}{\lim}\left( \int_{-\infty}^{\infty} f_p(t) dt - \sum_{x \in E, |x|>  R} f_p(x)\right),$$
and then 
$$\mathcal{N}_{B} = \underset{p \rightarrow \infty}{\lim}\left( \int_{-\infty}^{\infty} f_p(t) dt - \sum_{x \in E, |x|>  R} f_p(x)\right) - \mathcal{N}_{[-R,R] \backslash B}.$$
The right-hand side of the last expression is clearly in the $\sigma$-algebra $\Sigma_{E \backslash B}$. 

\bibliographystyle{halpha}
\bibliography{rigiditybib}
\end{document}